\documentclass[12pt]{amsart}
\input xy
\xyoption{all}
\newcommand{\SO}{{\mathcal{O}}}

\newcommand{\CC}{\mathbb{C}}

\newcommand{\Spec}{\operatorname{Spec}}

\newcommand{\Hom}{\operatorname{Hom}}
\newcommand{\Pic}{\operatorname{Pic}}
\newcommand{\Jac}{\operatorname{Jac}}
\newcommand{\Sym}{\operatorname{Sym}}
\newcommand{\too}{\longrightarrow}
\newcommand{\rk}{\operatorname{rk}}

\newcommand{\pdeg}{\text{\rm par-deg}\,}
\newcommand{\pmu}{\text{\rm par-}\mu \,}

\newcommand{\op}{\operatorname}
\newcommand{\PM}{\mathcal {PM}^{d}_{\alpha}} 
\newcommand{\PF}{\mathcal {PM}^{\Lambda}} 
\newcommand{\PHM}{\mathcal {PM}^{d}_{_{\text{Higgs}}}} 
\newcommand{\PHF}{\mathcal {PM}^{\Lambda}_{_{\text{Higgs}}}} 

\newtheorem{theorem}{Theorem}[section]
\newtheorem{lemma}[theorem]{Lemma}
\newtheorem{proposition}[theorem]{Proposition}
\newtheorem{remark}[theorem]{Remark}

\newcommand{\nc}{\newcommand}
\nc{\on}{\operatorname}
\nc\et{\on{Sch}_{et}}
\nc{\C}{{\mathcal Q}}
\nc\Tors{\on{Tors}}
\nc{\A}{{\mathcal A}}
\renewcommand{\P}{{\mathcal P}}
\nc{\Triv}{\op{Triv}}

\numberwithin{equation}{section}
\begin{document}

\title[SYZ duality for parabolic Higgs moduli]{SYZ duality for
parabolic Higgs moduli spaces}

\author[I. Biswas]{Indranil Biswas}

\address{School of Mathematics, Tata Institute of Fundamental
Research, Homi Bhabha Road, Bombay 400005, India}

\email{indranil@math.tifr.res.in}

\author[A. Dey]{A. Dey}

\address{Department of Mathematics, Indian Institute of Technology, Madras,
Chennai 600036, India}

\email{arijitdey@gmail.com}



\date{}

\begin{abstract}
We prove the SYZ (Strominger-Yau-Zaslow) duality for the moduli 
space of full flag parabolic Higgs bundles over a compact 
Riemann surface. In \cite{HT2}, the SYZ duality was proved for 
moduli spaces of Higgs vector bundles over a compact Riemann 
surface.
\end{abstract}

\maketitle

\section{Introduction}

\subsection{ Mirror symmetry and SYZ duality}
Mirror symmetry was discovered in the late 1980's by physicists studying 
superconformal field theories. 
Let $X$ be an $n$ dimensional complex Calabi-Yau manifold with a 
Ricci-flat K\"ahler form $\omega$ and 
a nowhere vanishing holomorphic $n$-form $\Omega$ on $X$. A submanifold 
$Z \,\subset\, X$ of real dimension $n$ is called Lagrangian if 
$\omega\vert_Z \, =\,0$; further a Lagrangian submanifold is said to be 
special if $({\rm Im}\, \Omega)\vert_Z \,=\, 0$. After simplifying a great deal, 
mirror symmetry is an one-to-one duality (in an appropriate sense) between two 
class of objects:
\begin{enumerate}
\item{} Pairs of the form $(Z\, ,L)$, where $Z$ is a holomorphic submanifold 
of $X$ and $L$ is a holomorphic line subbundle on $Z$
(such a pair is called a holomorphic $D$-brane).

\item{} Special Lagrangian $D$-branes, which is a pair 
$(\widehat{Z}\, ,\widehat{L})$ where $\widehat{Z}$ is a special 
Lagrangian 
submanifold of a certain Calabi-Yau manifold $\widehat X$ (mirror 
partner), and $\widehat L$ is a flat $U(1)$ line bundle on 
$\widehat{Z}$. 
\end{enumerate} 

Since any point $x\, \in\, X$ is a submanifold, it should correspond to a 
pair $(\widehat{Z}\, ,\Lambda)$, 
where $\widehat{Z}$ is a special Lagrangian submanifold of a 
fixed Calabi-Yau manifold $\widehat{X}$ and $\Lambda$ is a flat
$U(1)$-line bundle on $\widehat{Z}$. By a theorem of McLean, the 
deformation space 
for a special Lagrangian manifold $\widehat{Z}$ is unobstructed 
and is parametrized by 
$H^1(\widehat{Z}\, ,\mathbb R)$, hence 
$\dim H^1(\widehat{Z},\,\mathbb R) \,=\,\text{dim}_{\mathbb C}(X) \,=\,n$. 
The moduli space of flat $U(1)$ line bundles is given by the torus 
$H^1(\widehat{Z},\,\mathbb R) / H^1(\widehat{Z},\,\mathbb Z)$. This 
gives a hint that a moduli of special Lagrangian submanifolds of a fixed 
Calabi-Yau manifold should have a $n$-torus fibration over an affine base of 
real dimension $n$.

Motivated by this Strominger-Yau-Zaslow made a conjecture \cite{SYZ}.

\noindent
{\bf SYZ Conjecture:}\, If $X$ and $\widehat X$ are mirror pair of 
Calabi-Yau $n$-folds, then there exist fibrations $f:\, X \, 
\longrightarrow \, B$ and $\widehat{f}:\, \widehat{X} \, \longrightarrow \, 
B$ whose fibers are special Lagrangian such that the general fiber is an 
$n$-torus. Furthermore, these fibrations are dual, in the sense that 
canonically $X_b\,=\,H^1(\widehat{X_b},\, S^1)$ and $\widehat{X_b}\, =\, 
H^1(X_b,\,S^1)$, whenever the fibers $X_b$ and $\widehat{X_b}$ are 
non-singular tori.

\subsection{Reformulation of the SYZ conjecture in terms of unitary
gerbes} 

Hitchin introduced the notion of a flat unitary gerbe (known as 
$\mathcal B$-fields to physicists) and reformulated the SYZ conjecture in 
terms of this $\mathcal B$-fields \cite{Hi}. To make sense one needs a 
further assumption that the special 
Lagrangian fibers are linearly equivalent for both $X$ and 
$\widehat X$. Then by \cite[Theorem 3.3]{Hi}, the restriction map for the 
second cohomology $H^2(X,\,\mathbb R) \longrightarrow H^2(X_b,\,\mathbb R)$ 
is zero. This means that the restriction map $H^2(X,\,S^1)\,\longrightarrow
\,H^2(X_b,\,S^1)$ is trivial. So the flat unitary gerbe $\mathcal B$ has 
trivial holonomy on each torus fiber (hence trivial). Therefore, one should 
work with pairs $(X_b,\, T)$, where $X_b$ is a special Lagrangian submanifold 
and $T$ is a flat trivialization of the gerbe $\mathcal B$ on $Z$.

The modified mirror conjecture as proposed by Hitchin, \cite{Hi}, is 
the following:

\noindent
{\bf Conjecture:}\, The mirror of a Calabi-Yau manifold $X$ with a 
$\mathcal B$-field is the moduli space of pairs $(Z\, ,T)$, where $Z$ is a 
special 
Lagrangian submanifold of $X$ and $T$ is a flat trivialization of the gerbe 
$\mathcal B$ on $Z$.

Two Calabi-Yau $n$-orbifolds $M$ and $\widehat M$,
equipped with flat unitary gerbes $B$ and $\widehat B$
respectively, 
are said to be {\em SYZ mirror partners\/} if there is an
orbifold $N$ of real dimension $n$ 
and there are smooth surjections
$$
\mu\,:\,M \,\longrightarrow \,N\, ,~\, ~\, \widehat{\mu}\,:\,
\widehat{M}\, \longrightarrow\, N
$$
such that for every $x \,\in\, N$ which is a regular value of $\mu$ and 
$\widehat \mu$, the fibers $L_x := \mu^{-1}(x) \subset M$ 
and ${\widehat L}_x := \widehat{\mu}^{-1}(x) \,\subset\, \widehat M$ 
are special Lagrangian tori which are dual to each other in the
sense that there are smooth identifications
$$L_x \,=\, \text{Triv}^{U{1}} (\hat L_x, \hat B)~\, \text{ and }\, ~ 
{\widehat L}_x \,= \,\text{Triv}^{U{1}} (L_x, B)
$$
that depend smoothly on $x$.

\subsection{The result of Hausel and Thaddeus} 

The moduli spaces of Higgs bundles admit natural dual pairs of
hyper-K\"ahler integrable systems \cite{Hi1}, \cite{Hi2}. The 
hyper-K\"ahler metric 
and the collection of Poisson-commuting functions determining the integrable
system produce a family of special Lagrangian tori on the moduli spaces, which 
is a key requirement of SYZ conjecture. Moreover, the families of tori on the 
$\text{SL}(r,{\mathbb C})$ and $\text{PGL}(r,{\mathbb C})$ moduli spaces are 
dual in the appropriate 
sense, which is the other requirement of SYZ conjecture. 

This work of Hausel and Thaddeus was extended to principal $G_2$ in \cite{Hi4}.
In \cite{DP}, this was extended to all semisimple groups.
(See related works \cite{FW}, \cite{GW} and \cite{Wi}.)

In \cite{HT1}, Hausel-Thaddeus made an announcement that moduli 
spaces of $\text{SL}(r,{\mathbb C})$ and $\text{PGL}(r,{\mathbb C})$ parabolic 
Higgs bundles are 
mirror partner to each other (in the sense of SYZ) and their stingy E 
polynomials are same. In \cite{HT2}, they gave a proof of this conjecture 
in non-parabolic case. 

Our aim here is to address the case of parabolic vector bundles with complete
quasi-parabolic flags. We follow the proof of Hausel-Thaddeus; the key ingredient 
here is the identification of the parabolic Hitchin fiber as the Prym variety of 
a certain spectral cover (which was done in \cite{BM}, \cite{GL}). 

It should be clarified that the Higgs fields that we consider have 
nilpotent residue. The corresponding moduli space forms a symplectic leaf
of the moduli space of Higgs bundles for which the residue of the Higgs field
satisfies the weaker condition that it is only flag preserving.

\medskip
\noindent
\textbf{Acknowledgements.}\, We thank the referee for pointing out references. 
The second author would like to thank J. Martens for an useful discussion. He 
also thanks TIFR for hospitality while some part of this work was done. Both 
authors thank The Institute of Mathematical Sciences at Chennai for hospitality.

\section{Preliminaries}\label{sec2}

\subsection{Parabolic Higgs bundles} 

Let $X$ be an irreducible smooth projective curve over $\CC$ of 
genus $g$, with $g\,\ge\, 2$. Let
$D\, \subset\, X$ be a nonempty finite subset of $n$ points.

A \textit{quasi-parabolic structure}, over $D$, on a holomorphic 
vector bundle $E\, \longrightarrow\, X$ is a filtration
$$
E_x\,=\, E_{x,0}\,\supset\, E_{x,1}\,\supset\,\cdots\,\supset\, 
E_{x,r_x}\,\supset\, E_{x,r_x+1}\,=\, \{0\}
$$
for each $x\,\in\, D$. A \emph{parabolic structure} on $E$ is
a quasi-parabolic structure as above together with rational numbers
$$
0\,\le\,\alpha_{x,0}\,<\, \alpha_{x,1}\,<\, \cdots\,<\, 
\alpha_{x,r_x}\,<\, 1 \, ,
$$
which are called \textit{parabolic weights}.
A \emph{parabolic vector bundle} over $X$ of rank $r$ is a 
holomorphic vector bundle of rank $r$ on $X$ equipped with a 
quasi-parabolic structure over $D$ together 
with parabolic weights. The system of parabolic weights
$\{(\alpha_{x,0}\, ,\cdots\, , \alpha_{x,r_x})\}_{x\in D}$ will be denoted 
by $\alpha$.

For a parabolic vector bundle $E_*\,=\,\, 
(E,\{E_{x,i}\},\alpha_{*})$, the \emph{parabolic degree} is 
defined to be
$$
\pdeg(E_*)\, =\, \deg(E)+ \sum_{x\in D}\sum_{i=0}^{r_x} 
\alpha_{x,i} \, ,
$$
and the \emph{parabolic slope} is 
defined to be $\pmu(E_*)\,:=\,\pdeg(E_*)/rk(E)$. Any holomorphic 
subbundle $F$ of $E$ has a parabolic structure induced by the
parabolic structure on $E$; the resulting parabolic vector 
bundle will be denoted by $F_*$. A parabolic bundle is said to 
be \textit{stable} (respectively, \textit{semistable}) if for all 
holomorphic subbundles $F\,\subset\, E$ with $0\, <\, \rk(F)\, 
<\, rk(E)$,
\begin{equation}\label{mu}
\pmu(F)\,<\,\pmu(E)~\,~{\rm (respectively,~}\, \pmu(F)\,\leq\,\pmu(E)
{\rm )} \, .
\end{equation}

The moduli space of semistable parabolic vector
bundles of rank $r$ and degree $d$ with fixed parabolic data was 
constructed by Mehta and Seshadri, \cite{MS}, using Mumford's 
Geometric Invariant Theory. This moduli space, which we will 
denote by $\PM$,
is smooth for a generic choice of weights $\alpha$. We recall 
that given rank and degree, a system of parabolic weights 
$\alpha$ with multiplication is called \textit{generic} if the
semistability condition implies the stability condition. 

Consider the determinant morphism
\begin{equation}\label{det}
\det\,:\, \PM \, \longrightarrow \,\Jac^d(X)
\end{equation}
that sends a parabolic vector bundle $E_*$ to
$\bigwedge^{\rm top} E$. Choose $\Lambda 
\,\in\, 
\Jac^d(X)$, and define
\begin{equation}\label{pf}
\PF\,:=\, {\rm det}^{-1}(\Lambda)\, .
\end{equation}
So $\PF$ is a moduli space of twisted $\text{SL}(r,{\mathbb C})$--bundles
with parabolic structure (see \cite[Section 2]{BLS} for twisted 
$\text{SL}(r,{\mathbb C})$--bundles). For any other line bundle 
$\Lambda_1\, \in\, \text{Jac}^{d}(X)$, the morphism
$$
{\rm det}^{-1}(\Lambda)\, \longrightarrow\, 
{\rm det}^{-1}(\Lambda_1)
$$
that sends any parabolic vector bundle $E_*$ to $E_*\otimes
\zeta$, where $\zeta$ is a fixed $r$-th root of 
$\Lambda_1 \otimes \Lambda^{-1}$, is an isomorphism.
Thus the isomorphism class of the moduli space $\PF$ does not 
depend on the choice of $\Lambda \,\in\, \Jac^d(X)$.

The abelian variety $\Pic^0(X) \,=\, \Jac^0(X)$ acts on $\PM$
via
\[
(L,E_*) \,\longmapsto \, L {\otimes} E_*\, .
\] 
The quotient
$$
\widetilde{\PM}\,:=\, \PM /\Pic^0(X)\, ,$$
which exist as a projective variety by 
\cite{S}, is the component of the moduli space of parabolic 
$\text{PGL}(r,{\mathbb C})$--bundles corresponding to degree $d$ (the
connected components of the moduli space of parabolic
$\text{PGL}(r,{\mathbb C})$--bundles are irreducible).

Let
\begin{equation}\label{e1}
\Gamma \,:=\, \Pic^0(X)[r]
\end{equation}
be the group of $r$-torsion points of the Jacobian; it is
isomorphic to $({\mathbb Z}/r{\mathbb Z})^{2g}$, in particular, 
its order is $r^{2g}$. The action of
$\Gamma$ on $\PM$ preserves the subvariety
${\PF}$ defined in \eqref{pf}. Let
$$
\widetilde{\PF}\,:=\, {\PF}/\Gamma
$$
be the quotient; it is a projective variety by \cite{S}.

Let $K$ be the holomorphic cotangent bundle of $X$. The line bundle
$K\otimes \SO_{X}(D)$ will be denoted by $K(D)$. A \emph{parabolic
Higgs bundle} is a pair $(E_*\, ,\Phi)$, where $E_*$
is a parabolic vector bundle and
$$
\Phi\,:\, E \,\too\, E\otimes K(D)
$$
is a homomorphism
which is {\it strongly parabolic}, meaning 
\begin{eqnarray*}
\Phi(E_{x,i})\,\subset\, E_{x,i+1}\otimes K(D)_{x}
\end{eqnarray*}
for each point $x \,\in\, D$ and $i\, \in\, [0\, ,r_x]$. 

A parabolic Higgs bundle $(E_*\, ,\Phi)$ is called \emph{(semi)-stable} if
the slope condition in \eqref{mu} holds whenever $\Phi$ preserves
$F$. Let $\PHM$ denote the moduli space of semistable parabolic Higgs
bundles of rank $r$ and degree $d$ with the given parabolic data.
Let $$\PHF\, \subset \, \PHM$$ be the subvariety consisting of
all $(E_*\, \Phi)$ such that $\det (E)\, =\, \Lambda$ and
$\text{trace}(\Phi)\,=\, 0$.

For $E_* \, \in \PM$, 
\[
T_{E_*} \PM \,=\, H^1(X,\, \text{End} (E_*)) 
\]
\cite{Y}, \cite{Y1}, where $\text{End} (E_*)$ is the sheaf of 
endomorphisms of the underlying vector bundle $E$ preserving the
quasi-parabolic filtrations. Applying the parabolic 
analog of Serre duality, \cite[Section 3]{Y}, 
\[
T_{E_*} \PM \,=\, H^0(X,\, \text{SEnd} (E_*) \otimes K(D))^{*}
\, ,
\]
where $\text{SEnd}(E_*)\, \subset\, \text{End} (E_*)$ is the 
strongly parabolic endomorphisms. Hence the 
total space $T_{E_*}^{*}\PM$ of the cotangent bundle maps to
$\PHM$. This map is an open embedding.

The group $\Gamma$ acts on $\PHF$ via tensor product (the Higgs 
field does not change). The quotient $\PHF /\Gamma$ will be 
denoted by $\widetilde{\PHF}$. 

\subsection{Parabolic Hitchin system}\label{se2.2}

In this subsection we recall the Hitchin map and the spectral 
curve for a parabolic Higgs bundle (see \cite{BM}, \cite{GL}, 
\cite{LM} for details). 

For notational convenience, the line bundle $K^{\otimes a}
\otimes {\mathcal O}_X(bD)$ will be denoted by $K^aD^b$.

The parabolic Hitchin space is defined as
$$
\mathcal H \,= \, H^{0}(K^2D) \oplus \cdots 
\oplus H^0(K^rD^{r-1})\, .
$$
The characteristic polynomial or trace map of a Higgs field
defines the parabolic Hitchin map 
\begin{equation}\label{hitchin}
h\,:\, \PHF\,\longrightarrow \,\mathcal H
\end{equation}
It is known that this morphism $h$ is proper (see \cite{Y1}).

Let $Z\, :=\,\underline{\Spec}\Sym^\bullet (K^{-1}\otimes 
\SO_{X}(D)^{-1})$ be the total 
space of the line bundle $K(D)$ which is a quasi-projective surface.
Let
\begin{equation}\label{prp}
p:Z\, \too \,X
\end{equation}
be the natural projection. For $s \in \mathcal H$, there 
exists an algebraic curve, denoted as $X_s$, in $Z$ which is 
known as the {\it spectral curve}. We will very 
briefly recall it (for details see \cite{BM}, \cite{GL}). 

For any $(s_1,\, \cdots, \, s_{r-1}) \, \in \, \mathcal H$, 
consider the map $S$ from $Z$ to the total space of the line 
bundle $K^rD^r$ given by 
$$
z\,\longmapsto\, z^{\otimes r} + z^{\otimes (r-2)} \otimes 
s_1(p(z))
\, + \, \cdots + \, s_{r-1}(p(z)) \, \in \, (K^rD^r)_{p(z)}
$$
where $z\, \in\, Z$, and $p$ is the projection in \eqref{prp}. 
The inverse image $S^{-1}(0_X)\, \subset\, Z$, where $O_X\, 
\subset\,K^rD^r$ is the zero section, is the parabolic spectral 
curve associated to $s \,:=\,(s_1,\, \cdots,\,s_{r-1})$. This
parabolic spectral curve will be denoted by $X_s$. The 
restriction of $p$ to $X_s$ will again be denoted as $p$.

Henceforth, we assume that for each point $x\, \in\, D$, the
quasi-parabolic flag is complete. In other words, $r_x\,=\, r-1$.

There is a Zariski open dense subset ${\mathcal U} \,\subset \, 
\mathcal H$ such that for any $s \,\in\, \mathcal U$ the spectral 
curve $X_s$ is smooth and connected \cite[Lemma 3.1]{GL}; the 
assumption that the quasi-parabolic flags are complete is needed 
for this. The fiber over $s$ is isomorphic to
$$
\text{Prym}^{d'}(Y_{s}) \, = \, \{L \,\in \,{\rm Pic}^{d'}(Y_{s}) 
\, \mid \, 
\text{det}(\pi_*(L))\,=\,\xi \}\, ,
$$
where $d'\,=\,d + r(r-1)(n+2g-2)/2$ \cite[Lemma 3.2]{GL}. This is a 
$\Gamma$-invariant closed subvariety of $\PHF$.

Note that tensoring by a line bundle does not change the 
characteristic polynomial, hence the Hitchin map $h$ in 
\eqref{hitchin} descends 
down to 
\begin{equation}\label{hitchin1}
\widetilde{h}\,:\,\widetilde{\PHF} \, \longrightarrow 
\,\mathcal H\, .
\end{equation}
So $h$ is the composition of $\widetilde{h}$ with the quotient map
$\PHF\, \longrightarrow\,\widetilde{\PHF}$.

We will describe the fibers of the Hitchin maps in \eqref{hitchin} and 
\eqref{hitchin1}.
For any $s\,\in\,\mathcal U$, 
\begin{enumerate}
\item{} $P^{d'} \,:=\,h^{-1}(s) \, =\, \text{Prym}^{d'}(X_s) 
\,=\,\text{Nm}^{-1}({\mathcal O}_X(d'x))$, where 
$$Nm\,:\,\text{Pic}^d(X_s)\,\longrightarrow\, \text{Pic}^{d}(X)$$ 
is the norm map defined by $\mathcal O_{X_s}(\sum_i d_ix_i) \, 
\longmapsto \,\mathcal O_{X}(\sum_i d_i\pi(x_i))$.

\item{} $\widehat{P}^{d'}\,:=\,\widehat{h}^{-1}(s) \,=\, 
\text{Prym}^{d'}(X_s)/ \Gamma$ (see \eqref{e1} for $\Gamma$).
\end{enumerate}

The fiber $P^{d'}$ is a torsor for $$P^0\,:=\, 
\text{Nm}^{-1}({\mathcal O}_X)\, ,$$ and $\widehat{P}^{d'}$ is a 
torsor for $\widehat{P}^0\,=\, \text{Prym}^{0}(X_s)/ \Gamma$. Hence
$h\vert_{\mathcal U}$ and $\tilde{h}\vert_{\mathcal U}$ can be thought
of as $P^0$ and $\widehat{P}^0$ torsors respectively over $\mathcal U$.

Let $\mathcal G$ be an abelian group-scheme over $X$, and let
$\mathcal A$ be a $\mathcal G$--torsor. For any integer 
$n$, the $\mathcal G$--torsor on $X$ obtained by extending the
structure group of $\mathcal A$ using the endomorphism of
$\mathcal G$ defined by $z\, \longmapsto\, z^{n}$ will be
denoted by $({\mathcal A})^n$.

\begin{lemma}\label{power}
For any integer $d$, we have
\begin{enumerate}
\item{} $P^d \, \cong \, (P^1)^d$ as $P^0$ torsors over $\mathcal 
U$, and

\item{} $\widehat{P}^{d} \, \cong \, (\widehat{P}^{1})^d $ as 
$\widehat{P}^{0}$ torsors over $\mathcal U$.
\end{enumerate}
\end{lemma}

In section \ref{se4} we will show that $P^d$ and $\widehat{P}^{d'}$
are mirror partners for a certain choice of a ``B'' field.

\section{Picard category and Gerbes}\label{Picard and 
Gerbes}\label{sec3}

We briefly recall definition of {\it sheaf of categories} over 
a scheme (for details see \cite{DG}, \cite{DM}, \cite{Gi}).
Let $\text{Sch}_{et}(X)$ denote the category of all \'etale 
neighborhoods of a scheme $X$. A \textit{presheaf of 
categories} on $\et(X)$ is a contravariant functor $\mathcal Q$ 
which assigns to every object $U\longrightarrow X$ in $\et(X)$
a category $\C(U)$ and to every morphism $f:U_1\longrightarrow 
U_2$ in $\et(X)$ a functor $f^*_\C:\C(U_2)\longrightarrow 
\C(U_1)$. Moreover, for every composition 
$$U_1\,\stackrel{f}{\longrightarrow}\, U_2\, 
\stackrel{g}{\longrightarrow}\, U_3\, ,$$
there is a transformation $f^*_\C\circ g^*_\C \longrightarrow 
(g\circ f)^*_\C$ satisfying an obvious compatibility relation 
for three-fold compositions.

A presheaf $\C$ of categories on $\et(X)$ is said to be
a \textit{sheaf of categories} if the following two axioms hold:
\begin{enumerate}
\item For $U\longrightarrow X$ in $\et(X)$ and a pair of objects
$C_1,C_2\in \C(U)$, the presheaf of sets on $\et(U)$
that assigns to $f\,:\,U'\,\longrightarrow\, U$ the set
$$\Hom_{\C(U')}(f_\C^*(C_1),\, f_\C^*(C_2))$$ is a sheaf.

\item If $f\,:\,U'\,\longrightarrow\, U$ is a covering,
then the category $\C(U)$ is equivalent to
the category of descent data on $\C(U')$ with respect to $f$, meaning 
every descent data on $\C(U')$ with respect to $f$ is.
\end{enumerate}

A Picard category is a tensor category, in which every object is 
invertible. A basic example is the category of line bundles over a scheme.

A sheaf of categories $\P$ is said to be a sheaf of Picard categories if
for every $$(U\to X)\,\in\,\et(X)\, ,$$ $\P(U)$ is endowed with a 
structure of a
Picard category such that the pull-back functors $f^*_\P$ are
compatible with the tensor product in an appropriate sense. If $\P_1$
and $\P_2$ are two sheaves of Picard categories, one defines (in a
straightforward fashion) a tensor functor between them.

A category $Q$ is said to be a {\it gerbe over the Picard category}
$P$, if $P$ acts on $Q$ as a tensor category, and for any object 
$C\,\in\, Q$ the functor $P \longrightarrow Q$ given by 
$$
B\,\in\,P\,\Longrightarrow\, \on{Action}(P,\, C)\in Q
$$
is an equivalence.

Now, if $\P$ is a sheaf of Picard categories and $\C$ is another sheaf 
of categories we say that $\C$ is a {\it gerbe} over $\P$, if the 
following two conditions hold:
\begin{itemize}
\item For every $(U\to X)\,\in\,\et(X)$, $\C(U)$ has the structure of a 
gerbe over $\P(U)$. This structure is compatible with
the pull-back functors $f^*_\P$ and $f^*_\C$.

\item There exists a covering $U\longrightarrow X$ such that $\C(U)$ is 
non-empty.
\end{itemize}

The basic example of a gerbe over an arbitrary sheaf of Picard
categories $\P$ is $\P$ itself; it is called the trivial $\P$--gerbe.

Let $\A$ be a sheaf of abelian groups over $\et(X)$ which takes values in an 
abelian group $A$. For an object $f\,:\,U\,\longrightarrow\,X$ of $\et(X)$, let 
$\Tors_\A(U)$ denote the category of $\A\vert_U$--torsors on $U$. This is a 
Picard category, and the assignment $U\longrightarrow \Tors_\A(U)$ defines a 
sheaf of Picard categories on $\et(X)$ which we will call $\Tors_\A$ or 
$\A$-torsor. A gerbe over the sheaf of Picard categories $\A$-torsor will be 
called an $\A$--gerbe. Hence an $\A$--gerbe will be thought of as a 
torsor over the sheaf of Picard categories $\A$--torsors. 

An {\em isomorphism} between $\A$--gerbes is an equivalence of sheaves 
of categories as torsors over the sheaf of $\A$--torsors. The 
isomorphism classes of $\A$--gerbes are in one-to-one
correspondence with $H^2(X, \,\A)$ (see \cite{Gi}). 
 
A {\em trivialization} of an $\A$-gerbe is an isomorphism with the 
trivial gerbe $\A$--torsor. Two trivializations $z\, ,z'$ are {\em
equivalent} if the automorphism $z' \circ z^{-1}$ is given by
tensorization with a trivial $\A$-torsor. The space of equivalence
classes of trivializations of a trivial $\A$--gerbe $B$ denoted
$\Triv^{\mathcal A}(X,\,B)$ is an $H^1(X,\,{\A})$-torsor over a point \cite{Gi}. 

\begin{remark}\label{remark1}
{\rm Fix a short exact sequence $$0\,\longrightarrow \,\A\,\longrightarrow
\,\A'' \,\longrightarrow\,\A'\,\longrightarrow\, 0$$ of sheaves 
of groups (they need not be abelian) on $X$, and let $\tau_{\A'}$ be an 
$\A'$--torsor over $X$.
We introduce a sheaf of categories $\C\,=\,\C_{\tau_{\A'}}$ as follows:
For $U\,\in\,\et(X)$, let $\C(U)$ be the category of
all ``liftings'' of $\tau_{\A'}\vert_U$ to an $\A''\vert_U$--torsor. It is 
easy to check that $\C$ is a $\A$-gerbe over $X$.}
\end{remark}

\begin{remark}\label{remark2}
{\rm Let $\C_1$ be a $\P_1$-gerbe over $X$, and let ${\mathbf 
a}:\P_1\longrightarrow \P_2$ be a tensor functor of Picard category over 
$X$. Then one can construct a canonical induced $\P_2$--gerbe $\C_2$ 
over $X$ with the property that there exists a functor $\C_1\, 
\longrightarrow\, \C_2$,
compatible with the actions of $\P_1$ and $\P_2$ via ${\mathbf a}$}.
\end{remark}

Let $\C_1$ and $\C_2$ be two $\mathcal A$--gerbes over $X$.
Then $\C_1\times_X \C_2$ is an ${\mathcal A}\times_X {\mathcal 
A}$--gerbe over $X$. Consider the multiplication homomorphism
${\mathcal A}\times_X {\mathcal A}\, \longrightarrow\,
{\mathcal A}$. Let $\C_1\underset{\P}\otimes \C_2$ be the
$\mathcal A$--gerbe over $X$ given by $\C_1\times_X \C_2$ using
this homomorphism (see Remark \ref{remark2}). This
$\mathcal A$--gerbe $\C_1\underset{\P}\otimes \C_2$ is
called the {\it tensor product} of $\C_1$ and $\C_2$.

Now consider the inversion homomorphism
${\mathcal A}\, \longrightarrow\, {\mathcal A}$.
The $\mathcal A$--gerbe over $X$ given by $\C_1$ 
using this homomorphism (see Remark \ref{remark2}) will be 
denoted by $(\C_1)^{-1}$.

\medskip
\noindent
\textbf{Notation.}\, Let $\C$ be a $\mathcal A$--gerbe over $X$. 
For any positive integer $n$, the $n$-fold tensor product
$\C\underset{\P}\otimes \cdots \underset{\P}\otimes\C$ will be 
denoted by $(\C)^d$. For any negative integer $n$, the $n$-fold 
tensor product ${\C}^{-1}\underset{\P}\otimes \cdots 
\underset{\P}\otimes{\C}^{-1}$ will be denoted by $(\C)^{-d}$.
\medskip

\begin{remark}\label{remark3}
{\rm If $B$ is a trivial $\mathcal A$--gerbe
over $X$, then the tensor power $B^e$ for any $e \in \mathbb Z$ 
is also trivial, and, moreover, the set of all 
trivializations, which is a $H^1(X,\,{\mathcal A})$-torsor, has the 
following identification
\[
{\rm Triv}^{\mathcal A}(X, B^e)\, =\, {\rm Triv}^{\mathcal A}(X, 
B)^e
\] 
(it is an identification of $H^1(X,\,{\mathcal A})$-torsors).}
\end{remark}

\section{Trivializations and B fields}\label{se4}

Let $\mathcal U(1)$ (respectively, $\mathcal Z_{r}$) be the sheaf of 
abelian groups 
over ${\rm Sch}_{et}(\PHM)$ which takes values in $U(1)$
(respectively, $\mathbb Z_r$), and let 
$\mathcal {T}ors_{_{\mathcal U(1)}}$ (respectively, $\mathcal 
{T}ors_{_{\mathcal Z_r}}$) be the sheaf of Picard categories over $\PHM$ 
(see Section \ref{sec3} for definition). 

There is a universal parabolic Higgs vector bundle over $\PHM \times X$,
because the parabolic flags are complete \cite[p. 465, Proposition 
3.2]{BY}. Let $(\mathcal E,\Phi)$ be a Universal parabolic Higgs 
bundle on $\PHM \times X$. Restricting $\mathcal E$ to $\PHM \times 
\{c\}$, where $c\, \in\, X\setminus D$ is a fixed point, we get a 
vector bundle $\mathbb E$ on $\PHM$. 
Let $\mathbb P\, :=\,P(\mathbb E)$ be the associated projective bundle 
on $\PHM$ parametrizing line in the fibers of $\mathbb E$. From the exact 
sequence
$$
e\, \longrightarrow\, \mathbb Z_r\, \longrightarrow\,\text{SL}(r,{\mathbb C})
\, \longrightarrow\,\text{PGL}(r,{\mathbb C})\, \longrightarrow\, e
$$
it follows that the obstruction to lift the $\text{PGL}(r,{\mathbb C})$--bundle 
$\mathbb P$ to a $\text{SL}(r,{\mathbb C})$--bundle gives a class $B \,\in\, 
H^2(\PHM,\, {\mathcal Z}_r)$. This cohomology class $B$ corresponds to
the ${\mathcal Z}_r$--gerbe on $\PHM$ defined by the liftings of $\mathbb 
P$ to a $\text{SL}(r,{\mathbb C})$ bundle (see Remark \ref{remark1}).

\begin{lemma}\label{triviality}
The restriction of $B$ to each regular fiber $P^d$ of
the Hitchin map $h$ (see \eqref{hitchin}) is trivial as a 
$\mathcal{Z}_r$--gerbe.
\end{lemma}

\begin{proof}
Let $\mathcal L$ be a universal line bundle on $P^d \times X_s$
(see Section \ref{sec2} for $P^d$).
The projection of the spectral curve $X_s$ to $X$
will be denoted by $\pi$.
The push-forward $(\text{Id}\times \pi)_{*} \mathcal L$ is a 
vector bundle which admits a family of parabolic Higgs 
field inducing the inclusion $P^d \, \subset \, \PHM$. Hence we have 
${\mathbb P}((\text{Id}\times \pi)_* \mathcal{L})\vert_{P^{d} 
\times \{c\}}\,=\, 
\mathbb P$. Note that $\det 
((\text{Id}\times \pi)_*\mathcal{L})\vert_{P^{d}\times 
\{c\}}$ is isomorphic to $\xi\,=\, \otimes_{y \in \pi^{-1}(c)} 
\mathcal L\vert_{P^d \times \{y\}}$ (the points of $\pi^{-1}(c)$ 
are taken with multiplicities).

Consider the above line bundle $\xi$. Let $\eta$ be the $r$-th 
root of the line bundle $\xi^*$ on $P^d$, meaning $\eta^r\,= 
\,\xi^*$. Note that since the N\'eron--Severi class of
$\xi$ is divisible by $r$, such a line bundle $\eta$ exists.
It is easy to see that $(\text{Id}\times\pi)_*(\mathcal L\otimes 
p^{*}\eta)$, where $p$ is the projection of $P^d \times X_s$ to 
$P^d$, is a $\text{SL}(r,{\mathbb C})$ bundle on $P^{d} \times X$ such that 
${\mathbb P}((\text{Id}\times\pi)_*(\mathcal L \otimes 
p^{*}\eta))\,=\, \mathbb P$. Hence $B$ is a trivial 
${\mathcal Z}_r$--gerbe when restricted to $P^d$.
\end{proof}

As seen in the proof of Lemma \ref{triviality}, a trivialization 
of $B$ on $P^d$ is equivalent of giving a universal line bundle 
$\mathcal L\longrightarrow P^d\times X_s$ such that $\det 
(\text{Id}\times\pi)_* \mathcal L\vert_{P^d 
\times \{c\}}$ is trivial on $P^d \times \{
c \}$. Hence we have a natural identification of the set of trivializations of 
$B$ denoted as $\text{Triv}^{{\mathcal 
Z}_r}(P^d,B)$ with the set of isomorphism classes of such 
line bundles on $P^d \times X_s$; define
\[
T\,:=\,\{\mathcal L \rightarrow P^d \times X_s\, \mid \, 
{\mathcal L} 
\text{~is universal bundle with~} \text{det}({\rm Id} \times 
\pi)_{*}({\mathcal L})\vert_{P^d \times \{c\}} \,=\,\mathcal 
O_{P^d}\}\, .
\]
Note that this $T$ is naturally a ${\widehat P}^0[r]$--torsor 
since
$$\det (\text{Id}\times \pi)_*(\mathcal L\otimes p^* 
L)\,=\, \det (\text{Id}\times \pi)_*(\mathcal L)\otimes L^r\, .
$$ 

We have the following natural isomorphism
\[
H^1(P^d,\, {\mathcal Z}_r) \, =\, H^1(P^0,\, {\mathcal Z}_r) 
\,=\,{\widehat P}^0[r]\, .
\]
In terms of this isomorphism, the $H^1(P^d,\,{\mathcal 
Z}_r)$--torsor $\text{Triv}^{\mathbb Z_r}(P^d,B)$ gets
identified with the ${\widehat P}^0[r]$--torsor $T$. 
Using this identification, the set of trivialization of $B$ on 
$P^d$ will be considered as a $H^1(P^d,\,{\mathcal Z}_{r})$-torsor.

By Remark \ref{remark2}, any $\mathcal Z_{r}$--gerbe extends to 
a ${\mathcal U}(1)$--gerbe. Let $\mathcal B$ denote the
${\mathcal U}(1)$-gerbe given by the $\mathcal Z_r$--gerbe $B$. 
Since the $\mathcal Z_r$ gerbe $B$ on $P^d$ is trivial, the 
extended ${\mathcal U}(1)$--gerbe $\mathcal B$ on $P^d$ is also trivial. 
The set of all trivializations of $\mathcal B$ on $P^d$ as 
is denoted by $\text{Triv}^{{\mathcal U}(1)}(P^d, \mathcal B)$. This
$\text{Triv}^{{\mathcal U}(1)}(P^d, \mathcal B)$ is a
$H^{1}(P^d,\, {\mathcal U}(1))$--torsor. 

\begin{theorem}\label{ee}
For any $d\, ,e \,\in\, \mathbb Z$, there is an isomorphism of 
${\widehat P}^0$--torsors 
$$
{\rm Triv}^{{\mathcal U}(1)}(P^d,\,{\mathcal B}^e)
\,\stackrel{\sim}{\longrightarrow}\, {\widehat P}^e\, .$$ 
\end{theorem}

\begin{proof}
{}From Lemma \ref{power} and Remark \ref{remark3},
$$\text{Triv}^{{\mathcal U}(1)}(P^d,\mathcal B^e) \,\cong
\,(\text{Triv}^{{\mathcal U}(1)}(P^d,\mathcal B^1))^e \,\ 
\text{~and~}\,
\,\ \widehat{P}^e \, \cong \, (\widehat{P}^1)^e\, .$$
Therefore, it is 
enough to prove the theorem under the assumption that
$e \,=\,1$. So set $e \,=\,1$.

We have a natural identification as extension of scalers,
\[
\text{Triv}^{{\mathcal U}(1)}(P^d, \mathcal B) \,=\, 
\frac{\text{Triv}^{\mathcal Z_r}(P^d, B) \times H^1(P^d,\,
{\mathcal U}(1))} {H^1(P^d,\, {\mathcal Z}_{r})}\, .
\]
Under this identification, the above torsor
${\rm Triv}^{{\mathcal U}(1)}(P^d,\,{\mathcal B}^e)$ can be 
identified set theoretically with $\mathfrak T_1$ defined as 
follows:
$$
\{\mathcal L \longrightarrow P^d \times X_s\,\mid\, \, \mathcal L 
~\,
 \text{is a universal line bundle and}~ \, \mathcal L\vert_{P^d 
\times \{y\}} \, \in \text{Pic}^0({P^d})~ \,\forall\, y \in 
X_s\}\, .$$ 

We have a natural identification $\text{Pic}^{0}(P^d)\, \cong \, 
\frac{\text{Pic}^{0}(\widetilde{J^0})}{\text{Pic}^0{(J^0)}} 
\,=\, \frac{\widetilde{J^0}}{J^0}$ \cite[Lemma 2.2, Lemma 
2.3]{HT2}, hence $\mathfrak T_1$ can be 
identified 
with $\frac{\mathfrak T}{J^0}$ where $\mathfrak T$ is defined as 
follows:
$$
\{\mathcal L \longrightarrow \widetilde{J^d} \times X_s\,\mid \, 
\mathcal 
L ~\, \text{is a universal line bundle and}~ \, {\mathcal 
L}\vert_{\widetilde{J^d} \times \{y\}} \, \in\,
\text{Pic}^0{\widetilde{J^d}}\, \forall\, y\, \in X_s\}\, .$$

There is a natural isomorphism of $J^0$ with $\text{Pic}^0(J^0)$ 
given by the natural theta polarization on $J^0$. In terms of 
this identification, the action of $J^0$ on $\mathfrak T$ 
corresponds to the action of $\text{Pic}^0(J^0)$ defined by 
pull-back. Note that $\widehat{P}^d$ can be identified with 
$\frac{\widetilde{J^d}}{J^d}$. So it is enough to show that 
$\mathfrak T$ and $\widetilde{J^1}$ are isomorphic as 
$\widetilde{J^0}$-torsors.

The idea is to give two surjective set theoretic maps $f_1$ and 
$f_2$ from $X_s$ to these two torsors:
\[
\xymatrix{
& X_s \ar[ld]_{f_1} \ar[rd]^{f_2} & \\
\widetilde{J^1} && \mathfrak{T} 
}
\]
such that,
\begin{equation}\label{eql}
f_1(y') - f_1(y) \,=\, f_2(y) -f_2(y')\, .
\end{equation}
In view of this equality and the fact that both are 
$\widetilde{J^0}$-torsors, the identification of $f_1(y)$ with 
$f_2(y)$ gives the required isomorphism between $\widetilde{J^1}$
and $\mathfrak T$ as $\widetilde{J^0}$-torsors.

Now we will construct $f_1$ and $f_2$. The map 
$f_1$ is the Abel-Jacobi map which takes any $y \,\in\, X_s$ to 
the line bundle $\mathcal O_{X_s}(y)$ (which is an element of 
$\widetilde{J^1}$). The map $f_2$ sends any $y$ to the unique 
universal line bundle $\mathcal L$ on $P^d \times X_s$ satisfying 
the condition that ${\mathcal L}\vert_{P^d \times \{y\}} \, =\, 
\mathcal O_{ P^d}$. To show that
\eqref{eql} holds, we need the following: 
\begin{eqnarray}\label{eqlty}
(\mathcal L \otimes p^*(\mathcal O_{X_s}(-y'))\vert_{P^d \times 
\{y\}} \,= \, f_1(y) -f_1(y') \hspace{1cm} \text{for any}~\, y' 
\,\in\, X_s\, .
\end{eqnarray} 

Now, \eqref{eqlty} follows from two facts: Firstly, any universal 
bundle on $P^d \times X_s$ is of the form $p_2^{*}(L_0)\otimes 
F^{*}\mathcal P$, where $p_2$ is the projection to 
$X_s$, $L_0 \in \widetilde{J^d}$ is a fixed line bundle,
$\mathcal P$ is the universal line bundle on 
$\text{Pic}^0(\widetilde{J}^0) \times \widetilde{J}^0$ 
($=\,\widetilde{J}^0 \times \widetilde{J}^0$), 
and 
\[
F:\, \widetilde{J^d} \times X_s \,\longrightarrow\, 
\widetilde{J^0} \times \widetilde{J^0} 
\]
is defined by $(L,y) \,\longmapsto \, (L \otimes L_0^{-1},f_1(y) 
-f_1(y'))$.

Secondly, the involution of $\widetilde{J^0} \times 
\widetilde{J^{0}}$ exchanging the two factors takes the 
universal line bundle on $\widetilde{J^0} \times
\widetilde{J^{0}}\,=\, \text{Pic}^0(\widetilde{J}^0) \times
\widetilde{J}^0$ to its dual.
\end{proof}

Let
\begin{equation}\label{wtg0}
\widetilde \Gamma\,=\,\bigsqcup_{\gamma \in \Gamma} L_\gamma -\{0\}
\end{equation}
be the disjoint union of the total spaces of the nonzero 
vectors of the line bundles $L_\gamma$.
This has the structure of a group scheme over $X$ whose fiber at 
$x\,\in\, X$ is an abelian extension
\begin{equation}\label{wtg}
1 \longrightarrow {\mathbb C}^* \longrightarrow {\widetilde 
\Gamma}_c \longrightarrow \Gamma \longrightarrow 0\, .
\end{equation}

The group $\Gamma$ acts on $\PHM$; the action of any $L\, \in\,
\Gamma$ sends any $(E_*,\phi)$ to 
$(E_* \otimes L, \phi \otimes {\rm Id}_L)$; since the parabolic 
structure of $L$ is trivial, we may use the notation of the 
usual tensor product (note that $\phi \otimes {\rm Id}_L$ is a Higgs
field on $E_* \otimes L$ in a natural way). Since the
quasi-parabolic flags are complete, there exists an universal
parabolic Higgs bundle $(\mathcal E,\Phi)$ on $\PHM \times X$ 
\cite{BY}. In particular, $\mathcal E$ is a universal vector 
bundle on $\PHM \times X$ and 
$\Phi \, \in \, H^0(\mathbb{E}nd({\mathcal E}) \otimes p^*_X 
K(D))$, where $p_X$ is the projection of $\PHM \times X$ to $X$. 
Consider the projective bundle $P(\mathcal E)$ parametrizing the
lines in the fibers of $\mathcal E$. The group 
$\Gamma$ acts on $P(\mathcal E)$; the action of any $L\,\in\,
\Gamma$ sends any $((E_*\, ,\phi)\, ,\xi)$ to $((E_* 
\otimes L\, , \phi\otimes {\rm Id}_L)\, ,\xi \otimes L_y)$, where $\xi\, \in\, 
(E_*)_y$. Fix a point $c\, \in\, X$. Restricting $P(\mathcal E)$ 
to $\PHM \times \{c\}$ we get a $\Gamma$-equivariant projective bundle 
$\mathbb P$ on $\PHM$. The obstruction class to lift the 
$\Gamma$-equivariant $\text{PGL}(r,{\mathbb C})$-bundle $\mathbb P$ into a 
$\Gamma$-equivariant $\text{SL}(r,{\mathbb C})$-bundle gives a nontrivial 
$\Gamma$-equivariant gerbe 
${B} \,\in\, H^2_{\Gamma} (\PHM,\,{\mathcal Z}_r)$. Let
$\widetilde B$ be the $\mathcal Z_r$-gerbe $\widetilde{\PHM}$
forgetting the $\Gamma$-equivariant structure on $B$.

The following technical lemma which will be used in 
proving Lemma \ref{trivia}.

\begin{lemma}[{\cite[Lemma 3.3]{HT2}}]\label{crucial}
Let $\mathcal L \longrightarrow J^0 \times X$ be the universal 
line bundle which is trivial on $J^0 \times \{ c \}$. Then there 
is an action over $X$ of $\widetilde
\Gamma$ (constructed in \eqref{wtg0}) on the total space of 
$\mathcal L$, lifting the action of $\Gamma$ on $J^0$
by translation, so that the scalars ${\mathbb C}^*$ act with 
weight one on the fibers. 
\end{lemma}

\begin{proposition}\label{trivia}
The restriction of $\widetilde B$ to each regular fiber 
$\widetilde P^d$ of
the Hitchin map is trivial as a ${\mathcal Z}_r$-gerbe.
\end{proposition}

\begin{proof}
The statement that the restriction of $\widetilde B$ to each 
regular fiber $\widetilde P^d$ of 
the Hitchin map is trivial as a ${\mathcal Z}_r$-gerbe is 
equivalent to the statement that the projective bundle $\mathbb 
P\vert_{P^d}$ is the projectivization of a $\Gamma$-equivariant 
vector bundle on $P^d$. We have already seen in Lemma 
\ref{triviality} that $B$ is a
trivial gerbe on $P^d$. So ${\mathbb P}\vert_{P^d}$ is the 
projectivization of a vector bundle $V$ on $P^d$. Therefore, the
only thing to check is that the vector bundle $V$ can be chosen 
to be $\Gamma$ equivariant.

Recall how we got hold of a vector bundle $V$ on $P^d$; it came 
from a universal line bundle $\mathcal L$ on $P^d \times 
X_s$. Hence by Lemma \ref{crucial}, giving a 
$\Gamma$-equivariant vector bundle $V$ on $P^d$ is equivalent of 
giving a universal line bundle $\mathcal L$ over $P^d \times 
X_s$ on which 
$\widetilde{\Gamma}$ acts such that the scalers ${\mathbb C}^{*} 
\, \subset \,\widetilde{\Gamma}_{c}$ (see \eqref{wtg}) act with 
weight one and $\text{det}(\pi_{*}(\widetilde{L}\mid_{P^d \times 
\{c\}}))$ is in $\text{Pic}^0_{\Gamma}(P^d)$. The rest of the
proof will be devoted in showing the existence of such 
a universal line bundle $\mathcal L$ on $P^d \times X_s$. 

Let $\widetilde {L}$ be any universal bundle on $\widetilde 
J^{d} \times X_s$ and $\mathcal L$ be the Poincar\'e universal 
bundle on $J^0 \times X$ rigidified using $c$. We have following 
natural projection maps, 
\begin{eqnarray}
\xymatrix{
& P^d \times J^0 \times X_s \ar[r]^{\widetilde{p_{12}}} \ar[dl]_{p_{12}} 
\ar[d]^{p_{23}} \ar[dr]^{p_{13} } & \frac{P^d \times J^0}{\Gamma} \times X_s\\
P^d \times J^0 & J^0 \times X_s \ar[d]^{\text{Id} \times \pi} & P^d \times X_s 
\\
& J^0 \times X .
}
\end{eqnarray}
The action of the group $\Gamma$ on $P^d \times J^0$ is given 
by $$\gamma \cdot (L\, , M) \,\longmapsto\, (L \otimes 
L_{\gamma}\, , M 
\otimes L_{\gamma})\, .$$ By \cite[Lemma 2.2]{HT2}, we 
have an identification $\frac{P^d \times J^0}{\Gamma} \,=\, 
\widetilde{J^d}$. The pulled back line bundle
 $\widetilde{p_{12}}^{*}(\widetilde L)$ has a natural $\Gamma$ 
action; this action can be extended to an action of 
$\widetilde{\Gamma}$ on $\widetilde{p_{12}}^{*}(\widetilde L)$ 
by making $\mathbb C^{*}$ act trivially. Define
\[
\mathcal M \,:=\, \widetilde{p_{12}}^*(\widetilde L)\otimes 
p_{12}^{*}({\widetilde L}^{-1}) \otimes p_{23}^{*} 
(\text{Id}\times \pi)^{*}\mathcal L\, .
\]

Note that for any $x$, $y$ and $z$ in $P^d$, $J^0$ and $X_s$ 
respectively, we have 
\[
\mathcal M\vert_{\{x\}\times J^0 \times X_s} \,= \,{\mathcal
M}\vert_{ P^d \times \{y\} \times X_s} \, =\, {\mathcal 
M}\vert_{P^d \times J^0 \times \{z\}} \,=\, \mathcal O\, .
\]
Hence by the theorem of the cube 
(see \cite[p. 87, Theorem]{Mu}), the line bundle 
$\mathcal M$ is trivial, 
equivalently, there is a $\widetilde{\Gamma}$--equivariant 
isomorphism 
\begin{eqnarray}\label{equality}
\widetilde{p_{ 12}}^{*}\widetilde L \,=\,p_{12}^{*}(\widetilde 
L) \otimes p_{23}^{*}(\text{Id}\times \pi)^{*} (\mathcal 
L^{-1})\, .
\end{eqnarray}
Since the universal Poincar\'e line bundle $\mathcal L$ is 
trivial on 
$J^{0} \times \{c\}$, by Lemma \ref{crucial}, the group scheme 
$\widetilde \Gamma$ acts on $p_{23}^{*}(\text{id}\times \pi)^{*} 
(\mathcal L^{-1})$ with the scalers $\mathbb C^{*}$ acting 
by weight $-1$. From the isomorphism in \eqref{equality} we get 
that $\widetilde{\Gamma}$ acts on $p_{12}^{*}(\widetilde L)$ 
with the scalers $\mathbb C^{*}$ acting by weight one. 
Restricting $p_{12}^{*}(\widetilde L)$ to any base point of 
$J^0$ we get a $\widetilde \Gamma$-action on $\widetilde L$; 
this produces a $\Gamma$-action on $\widetilde L$ which we 
wanted. So we can make $\widetilde L$ in the proof of Lemma 
\ref{triviality} to be a $\Gamma$-equivariant line bundle.
\end{proof}

The set of all $\Gamma$-equivariant trivializations 
$\text{Triv}^{\mathcal Z_r}(\widehat{P}^d,\widetilde{B})$ can be identified 
with the set
$\widetilde{ T}$ defined by
\[
\{\widetilde{ L} \longrightarrow P^d \times X_s\,\mid \, 
\widetilde{L} 
~ \text{is a $\widetilde{\Gamma}$-equivariant such that}~ \, 
\text{det}({\rm Id}\times\pi )_{*}(\widetilde{L})\vert_{P^d 
\times \{c\}} \,=\,\mathcal O_{P^d}\}\, .
\]
Note that in $\widetilde{T}$ we are interested only those 
$\widetilde{\Gamma}$-actions on $\widetilde{L}$ such that the 
action of $\mathbb C^{*}$ is of weight one,
because by Lemma \ref{crucial} this will 
ensure that $\widetilde{L}\vert_{P^d \times \{y\}}$ is a 
$\Gamma$-equivariant line bundle on $P^d$. It is easy to see 
that $\widetilde{T}$ is a $\Pic^0_{\Gamma} (P^d)[r]$-torsor. 
Note that we have the following identifications
\[
\text{Pic}^{0}_{\Gamma}(P^d)[r]\,=\,\text{Pic}^{0}(P^d/\Gamma)[r]
\,=\, {\rm Pic}^{0}(P^0/\Gamma)[r] \,=\,P^0[r]\, ,
\]
where the last equality follows from the fact that the dual of 
$P^0$, namely $H^1(P^0,\,{\mathcal U}(1))$, is $P^0/\Gamma$ \cite[Lemma 
2.3]{HT2}. Let $\widetilde{\mathcal B}$ denote the extended 
${\mathcal U}(1)$-gerbe given by the $\mathcal{Z}_{r}$-gerbe (see Remark 
\ref{remark2}) which is also trivial over $\widehat{P}^d$. 

\begin{theorem}\label{propn2}
For any $d,e \in \mathbb Z$, there 
is a smooth isomorphism of $P^0$-torsors
$${\rm Triv}^{{\mathcal U}(1)}({\widehat P}^d,\widehat{\mathcal 
B}^e) \,\cong\, P^e\, .$$
\end{theorem}

\begin{proof}
As in the proof of Theorem \ref{ee} we can assume $e\,=\,1$.
The set of all trivialization of $ \widetilde{\mathcal B}$, 
which is denoted as $\text{Triv}^{{\mathcal U}(1)}(\widehat{P}^d\, , 
\widetilde{\mathcal B})$ 
is a $H^1(P^d/\Gamma,\, {\mathcal U}(1))\,(=P^0)$-torsor. This 
$\text{Triv}^{{\mathcal U}(1)}(\widehat{P}^d\, ,
\widetilde{\mathcal B})$ can be identified with 
$\widetilde{\mathfrak T}$ defined by
\[
\{\widetilde{\mathcal L} \longrightarrow P^d \times X_s\,\mid \, 
\widetilde{\mathcal L} ~\,\text{is a 
$\widetilde{\Gamma}$-equivariant such that}~ \, \text{det}(\text{Id}
\times \pi)_{*}(\widetilde{\mathcal L})\vert_{P^d \times \{c\}} 
\, \in\, \text{Pic}^0_{\Gamma}(P^d)\}\, .
\]
Note that in $\widetilde{\mathfrak T}$ we are interested in only those 
$\widetilde{\Gamma}$ action on $\widetilde{L}$ such that the action of $\mathbb 
C^*$ is of weight one so that $\widetilde{\mathcal L}\mid_{P^d \times \{y\}}$ 
is a $\Gamma$-equivariant line bundle on $P^d$. Note that the difference between 
$\widetilde T$ and $\widetilde{\mathfrak T}$ is that in $\widetilde T$ we 
require the determinant of $\widetilde{L}$ is trivial, while in 
$\widetilde{\mathfrak T}$ we require the determinant to be a
$\Gamma$-equivariant line bundle on $P^d$ of degree zero.

In the rest of the proof we will show that the $P^0$-torsor 
$\widetilde{\mathfrak T}$ is isomorphic to the $P^0$-torsor $P^1$. 

First note that $P^1$ sits naturally inside $\widetilde{J^1}$ as a
$P^0$-torsor. 

In the proof of Theorem \ref{ee} we have seen that the set of all 
universal line bundles $\widetilde{\mathcal L}\, \longrightarrow \, 
\widetilde{J^d} \times X_s$ with $\widetilde{\mathcal 
L}\vert_{\widetilde{J^d}\times \{c\}} \, \in \, 
\text{Pic}^0(\widetilde{J^d})$ is isomorphic to $\widetilde{J^1}$ as a 
$\widetilde{J^0}$-torsor. We will use this fact to give 
an inclusion of $\widetilde{\mathfrak T}$ into $\widetilde{J^1}$ 
compatible with respect to $P^0 \,\subset \,\widetilde{J^0}$. Send any 
$\widetilde{\mathcal L} \, \in \, \widetilde{\mathfrak T}$ to 
$p_{13}^{*}(\widetilde{\mathcal L})\otimes p_{23}^{*}(\text{Id} \times 
\pi)^{*}(\mathcal L)^{-1}$. Since $\mathbb C^{*}$ acts on 
$p_{13}^{*}(\widetilde{\mathcal L})$ and $p_{23}^{*}(\text{id} \times 
\pi)^{*}(\mathcal L)^{-1}$ by weights +1 and -1 respectively, the above 
tensor product is a $\widetilde{\Gamma}$-equivariant line bundle with 
scalers acting trivially, or in other words, it is a $\Gamma$-equivariant 
line bundle on $P^d \times J^0 \times X_s$. Hence 
$p_{13}^{*}(\widetilde{\mathcal L})\otimes p_{23}^{*}(\text{Id} \times
\pi)^{*}(\mathcal L)^{-1}$ descends down to 
$$\frac{P^d\times J^0}{\Gamma} \times X_s\,=\, \widetilde{J^d} \times X_s\, 
.$$
One can check easily that
$p_{13}^{*}(\widetilde{\mathcal L}) \otimes p_{23}^{*}(\text{Id} \times
\pi)^{*}(\mathcal L)^{-1}\vert_{\widetilde{J^d} \times \{c\}} \, \in \,
\text{Pic}^0(\widetilde{J^d})$, and the resulting map
$$
\widetilde{\mathfrak T}\, \longrightarrow\, \widetilde{J^1}\, ,~\,~\,
\widetilde{\mathcal L}\, \longmapsto\, 
p_{13}^{*}(\widetilde{\mathcal L})\otimes p_{23}^{*}(\text{Id} \times
\pi)^{*}(\mathcal L)^{-1}
$$
is injective.

So $\widetilde{\mathcal T}$ and $P^1$ are now both $P^0$-subtorsors of 
$\widetilde{J^1}$. The quotient by either is the constant torsor $J^0$
(the Jacobian of $C$). Therefore, the image of one in the quotient by the other 
gives a morphism from the base $U$ to $J^0$. But $U$ is a Zariski open set in an 
affine space, so its only morphisms to an abelian variety are the constant ones. 
Indeed, any nonconstant morphism from a Zariski open subset of ${\mathbb A}^1$ 
to an abelian variety extends to a nonconstant morphism from ${\mathbb 
P}^1_{\mathbb C}$ to the abelian variety. But there is no such map.
\end{proof}

From Theorem \ref{ee} and Theorem \ref{propn2} we conclude that 
over a generic open subset $U \, \subset \, \mathcal H$ the 
Hitchin fibers are SYZ mirror partners in the sense of Hitchin.

\end{document}